\documentclass[12pt]{article}
\usepackage{a4}

\usepackage{amsthm}
\usepackage{amsmath}
\usepackage{graphicx}
\usepackage{enumitem}
\setlist[enumerate,1]{label=(\roman*)}
\setlist{noitemsep}

\newtheorem{theorem}{Theorem}
\newtheorem{prop}[theorem]{Proposition}
\newtheorem{lemma}[theorem]{Lemma}
\newtheorem{coro}[theorem]{Corollary}

\newcommand{\sm}{\setminus}
\newcommand{\BB}{\mathcal B}
\newcommand{\II}{\mathcal I}

\title{\textbf{An update on non-Hamiltonian $\mathbf{\frac{5}{4}}$-tough maximal planar graphs}
}
\author{Adam Kabela\thanks{Department of Mathematics, Institute for Theoretical Computer Science,
and European Centre of Excellence NTIS,
University of West Bohemia, Pilsen, Czech Republic. Email: \texttt{kabela@ntis.zcu.cz}.}
}

\date{}

\begin{document}
\maketitle

\begin{abstract}
 Studying the shortness of longest cycles in maximal planar graphs,
 we improve the upper bound on the shortness exponent of the class of
 $\frac{5}{4}$-tough maximal planar graphs presented by
 Harant and Owens [Discrete Math. 147 (1995), 301--305].
 In addition, we present two generalizations of a similar result of
 Tk\'{a}\v{c} who considered $1$-tough maximal planar graphs [Discrete Math. 154 (1996), 321--328];
 we remark that one of these generalizations gives a tight upper bound.
 We fix a problematic argument used in the first paper.
\end{abstract}

%%%%%%%%%%%%%%%%%%%%%%%%%%%%%%%%%%%%%%%%%%%%%%%%%%%%%%%%%%%%%%%%%%%%%%%%%%%%%%%%%%%%%%%%%%%%%%%
%%%%%%%%%%%%%%%%%%%%%%%%%%%%%%%%%%%%%%%%%%%%%%%%%%%%%%%%%%%%%%%%%%%%%%%%%%%%%%%%%%%%%%%%%%%%%%%

\section{Introduction}

We continue the study of non-Hamiltonian graphs
with the property that removing an arbitrary set of vertices
disconnects the graph into a relatively small number of components (compared to the size of the removed set).
In the present paper, we construct families of maximal planar such graphs whose longest cycles are
short (compared to the order of the graph).

More formally, the properties which we study are
the toughness of graphs and the shortness exponent of classes of graphs
(both introduced in 1973).
We recall that following Chv{\'a}tal~\cite{chva}, the \emph{toughness} of a
graph $G$ is the minimum, taken over all separating sets $X$ of
vertices in $G$, of the ratio of $|X|$ to $c(G - X)$
where $c(G - X)$ denotes the number of components of the graph $G - X$.
The toughness of a complete graph is defined as being infinite.
We say that a graph is \emph{$t$-tough} if its toughness is at least $t$.

Along with the definition of toughness,
Chv{\'a}tal~\cite{chva} conjectured that there is a constant $t_0$
such that every $t_0$-tough graph
(on at least three vertices) is Hamiltonian.
As a lower bound on $t_0$,
Bauer et al.~\cite{74} presented
graphs with toughness arbitrarily close to $\frac94$
which contain no Hamilton path (and thus, they are non-Hamiltonian).
While remaining open for general graphs, Chv\'{a}tal's conjecture was
confirmed in several restricted classes of graphs; and also various
relations among the toughness of a graph and properties of its cycles are known.
We refer the reader to the extensive survey on this topic~\cite{surv}.

Clearly,
every graph (on at least five vertices) of toughness greater than $\frac{3}{2}$ is $4$-connected,
so every such planar graph is Hamiltonian
by the classical result of Tutte~\cite{Tutte}.
On the other hand, Harant~\cite{hara} showed that not every $\frac{3}{2}$-tough planar graph is 
Hamiltonian; and furthermore, 
the shortness exponent of the class of
$\frac{3}{2}$-tough planar graphs is less than $1$.

We recall that following Gr{\"u}nbaum and Walther~\cite{sef},
the \emph{shortness exponent} of a
class of graphs $\Gamma$ is the $\liminf$,
taken over all infinite sequences $G_n$ of non-isomorphic graphs of $\Gamma$
(for $n$ going to infinity),
of the logarithm of the length of a longest cycle in $G_n$
to base equal to the order of $G_n$.

Introducing this notation, Gr{\"u}nbaum and Walther~\cite{sef}
also presented upper bounds on the shortness exponent
for numerous subclasses of the class of $3$-connected planar graphs.
Furthermore, they remarked that the upper bound for the class of $3$-connected planar graphs itself
was presented earlier by Moser and Moon~\cite{momo} who used a slightly different notation.
Later, Chen and Yu~\cite{3-conn} showed that
every $3$-connected planar graph $G$
contains a cycle of length at least $|V(G)|^{\log_3 2}$;
in combination with the bound of~\cite{momo},
it follows that the shortness exponent of this class equals $\log_3 2$.
A number of results considering the shortness exponent and similar parameters
are surveyed in~\cite{owens}.

Considering the class of maximal planar graphs under a certain toughness restriction, 
Owens~\cite{owensTatra} presented non-Hamiltonian maximal planar graphs of
toughness arbitrarily close to $\frac{3}{2}$. 
Harant and Owens~\cite{5/4} argued that the shortness exponent of the class
of $\frac{5}{4}$-tough maximal planar graphs is at most $\log_9 8$.
Improving the bound $\log_7 6$ presented by Dillencourt~\cite{dill},
Tk\'{a}\v{c}~\cite{tkac} showed that it is at most $\log_6 5$
for the class of $1$-tough maximal planar graphs.

In the present paper, we show the following.

\begin{theorem}\label{main}
Let $\sigma$ be the shortness exponent of the class of maximal planar graphs
under a certain toughness restriction.
\begin{enumerate}
\setlength{\itemsep}{3pt}
\item If the graphs are $\tfrac{5}{4}$-tough, then $\sigma$ is at most $\log_{30}{22}$.
\item If the graphs are $\tfrac{8}{7}$-tough, then $\sigma$ is at most $\log_{6}{5}$.
\item If the toughness of the graphs is greater than $1$, then $\sigma$ equals $\log_{3}{2}$.
\end{enumerate}
\end{theorem}

We note that $\log_{9}{8} > \log_{30}{22}$,
that is, the statement in item $(i)$ of Theorem~\ref{main} improves the result of~\cite{5/4}.
Furthermore,
items $(ii)$ and $(iii)$
provide two different generalizations of the result of~\cite{tkac}
since $\frac{8}{7} > 1$ and $\log_{6}{5} > \log_{3}{2}$.

We remark that we fix a problem in a technical lemma presented in~\cite[Lemma~1]{5/4}.
The fixed version of this lemma (see Lemma~\ref{constr2}) is applied to prove the present results.

%%%%%%%%%%%%%%%%%%%%%%%%%%%%%%%%%%%%%%%%%%%%%%%%%%%%%%%%%%%%%%%%%%%%%%%%%%%%%%%%%%%%%%%%%%%%%%%
%%%%%%%%%%%%%%%%%%%%%%%%%%%%%%%%%%%%%%%%%%%%%%%%%%%%%%%%%%%%%%%%%%%%%%%%%%%%%%%%%%%%%%%%%%%%%%%

\section{Structure of the proof}
\label{s2}

In order to prove Theorem~\ref{main},
we shall construct three families of graphs whose properties are summarized in the following proposition.

\begin{prop}\label{shortAll}
For every $i = 1,2,3$ and
every non-negative integer $n$,
there exists a maximal planar graph $F_{i,n}$ on $f_i(n)$ vertices
whose longest cycle has
$c_i(n)$ vertices where 
\begin{enumerate}
\setlength{\itemsep}{3pt}
\item $f_1(n) = 1 + 101(1 + 30 + \dots + 30^n)$ and $c_1(n) = 1 + 93(1 + 22 + \dots + 22^n)$
and $F_{1,n}$ is $\frac{5}{4}$-tough,
\item $f_2(n) = 1 + 14(1 + 6 + \dots + 6^n)$ and $c_2(n) = 1 + 13(1 + 5 + \dots + 5^n)$
and $F_{2,n}$ is $\frac{8}{7}$-tough,
\item $f_3(n) = 4 + 5(1 + 3 + \dots + 3^n)$ and $c_3(n) = 3 \cdot 2^{n+3}-9n-15$
and the toughness of $F_{3,n}$ is greater than $1$.
\end{enumerate}
\end{prop}

Before constructing the graphs $F_{1,n}$,
we point out that the use of Proposition~\ref{shortAll}
leads directly to
the main results of the present paper. 
\begin{proof}[Proof of Theorem~\ref{main}]
We consider an infinite sequence of non-isomorphic graphs $F_{1,n}$
given by item $(i)$ of Proposition~\ref{shortAll};
and we recall that they are $\frac{5}{4}$-tough maximal planar graphs.
Furthermore, we have
\begin{equation*}
f_1(n) = 1 + \tfrac{101}{29}(30^{n+1} - 1)
\quad
\textnormal{and}
\quad
c_1(n) = 1 + \tfrac{93}{21}(22^{n+1} - 1).
\end{equation*}
It follows that
\begin{equation*}
\lim_{n\to\infty} \log_{f_1(n)} c_1(n) = \log_{30}{22}.
\end{equation*}
Thus, the considered shortness exponent is at most $\log_{30}{22}$.

Using similar arguments and considering items $(ii)$ and $(iii)$ of Proposition~\ref{shortAll},
we obtain the desired upper bounds.

Clearly, if $G$ is a maximal planar graph (on at least four vertices), then it is $3$-connected.
By a result of~\cite{3-conn}, $G$ contains a cycle of length at least $|V(G)|^{\log_3 2}$.
In combination with the upper bound obtained due to item
$(iii)$ of Proposition~\ref{shortAll}, we obtain that
for the class of maximal planar graphs of toughness greater than $1$,
the shortness exponent equals $\log_3 2$.
\end{proof}

In the remainder of the present paper, we construct the families of graphs
having the properties described in Proposition~\ref{shortAll}.
Basically, we proceed in four steps.
First, we introduce relatively small graphs $F_{i, 0}$ called `building blocks' in Section~\ref{sb},
and we observe key properties of their longest cycles.
We use these building blocks to construct larger graphs $F_{i, n}$ in Section~\ref{sc},
and we show that their longest cycles are short.
In Section~\ref{tb}, we study the toughness of the building blocks.
The toughness of the graphs $F_{i, n}$ is shown
in Sections~\ref{sg} and~\ref{tt}.

We remark that the graphs
$F_{1, n}$ and $F_{2, n}$
are obtained using a standard construction for bounding the shortness exponent
(see for instance~\cite{sef}, \cite{dill}, \cite{5/4}, \cite{tkac} or~\cite{chpl});
the improvement of the known bounds comes with the choice of suitable building blocks.
In addition, we formalize the key ideas of this construction to make them more accessible for further usage.

The construction of graphs $F_{3, n}$
can be viewed as a simple modification of the construction used in~\cite{momo}
(yet the toughness and longest cycles of the constructed graphs are different).

%%%%%%%%%%%%%%%%%%%%%%%%%%%%%%%%%%%%%%%%%%%%%%%%%%%%%%%%%%%%%%%%%%%%%%%%%%%%%%%%%%%%%%%%%%%%%%%
%%%%%%%%%%%%%%%%%%%%%%%%%%%%%%%%%%%%%%%%%%%%%%%%%%%%%%%%%%%%%%%%%%%%%%%%%%%%%%%%%%%%%%%%%%%%%%%

\section{Building blocks}
\label{sb}

We start by considering the graph $T$ depicted in Figure~\ref{f:54}
which plays an important role in the latter constructions.
We let $o_1$, $o_2$, and $o_3$ be its vertices of degree $6$.

A \emph{$T$-region} of a graph $G$ is an induced subgraph isomorphic to $T$
with a distinction of vertices referring to $o_1$, $o_2$, $o_3$ as to \emph{outer} vertices
and to the remaining vertices as to \emph{inner} vertices,
and with the property that no inner vertex is adjacent to a vertex of $G - T$.
Similarly, we define an \emph{$H$-region} for a given graph $H$ and a given distinction of its vertices.

We view the $T$-regions as replacing triangles of a graph with copies of $T$ in the natural way.
The basic idea of the present constructions is that if these triangles share many vertices,
then every cycle in the resulting graph misses many vertices.
We formalize this idea in Proposition~\ref{T}.

We recall that a vertex is called \emph{simplicial} if its neighbourhood induces a complete graph.

\begin{prop}\label{T}
Let $R$ be a $T$-region of a graph $G$ and let $C$ be a cycle containing 
all three simplicial vertices of $R$ and a vertex of $G - R$.
Then the subgraph of $C$ induced by the vertices of $R$ is a path containing all outer vertices of $R$;
two of them as its ends.
\end{prop}
\begin{proof}
Clearly, $R$ has three simplicial vertices none of which is an outer vertex.
The statement follows from the fact that every simplicial vertex has only one neighbour
which is not an outer vertex.
\end{proof}

\begin{figure}[ht]
    \centering
    \includegraphics[scale=0.6]{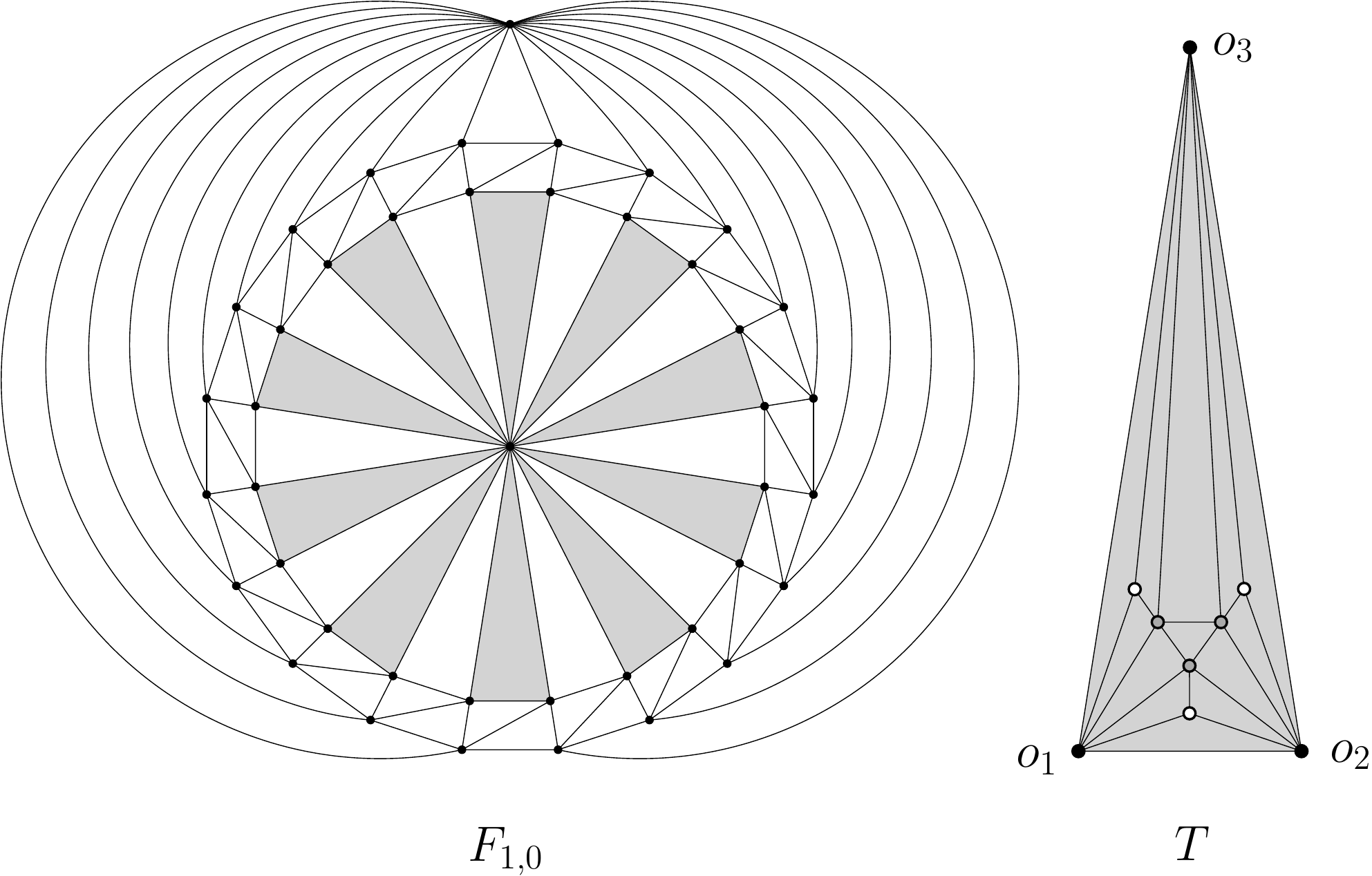}
    \caption{The graph $T$ and the construction of the graph $F_{1,0}$.
    The graph $F_{1,0}$ is obtained by replacing each of the highlighted triangles
    (of the graph on the left)
    with a copy of $T$ in the natural way
    (by identifying the vertices of the highlighted triangle
    with vertices $o_1, o_2, o_3$ of $T$).} 
    \label{f:54}
\end{figure}

Aiming for the graph $F_{1,0}$ (the building block for constructing the graphs $F_{1,n}$),
we consider a graph with a number, say $r$, of $T$-regions
which all share one outer vertex and which are otherwise disjoint.
By Proposition~\ref{T}, every cycle in this graph contains at most $2r + 2$
of the $3r$ simplicial vertices belonging to these $T$-regions.
(We view the building block used in~\cite{5/4} simply as choosing $r = 3$.)
With hindsight, we remark that the used construction leads to
the upper bound $\log_{3r}{(2r+2)}$ on the shortness exponent;
so we minimize this function over all integers $r \geq 3$,
that is, we choose $r = 10$.

We let $F_{1,0}$ be the graph depicted in Figure~\ref{f:54},
and we note that $F_{1,0}$ is a maximal planar graph.
Clearly, $F_{1,0}$ has $30$ simplicial vertices, and we colour these vertices white.
We recall that every cycle in $F_{1,0}$ contains at most $22$ white vertices.

Furthermore, we let $F_{2, 0}$ be the graph depicted in Figure~\ref{f:87}.
Clearly, $F_{2, 0}$ is a maximal planar graph having $6$ simplicial vertices;
and we colour these vertices white.
We use Proposition~\ref{T}, and we observe that a cycle in $F_{2, 0}$
contains at most $5$ white vertices.
Lastly, we let $F_{3,0}$ be the graph $T$.

For every $i = 1,2,3$, we define the outer face of $F_{i,0}$
as given by the present embedding (see Figures~\ref{f:54} and~\ref{f:87}).
In the following section, we shall use the blocks $F_{i,0}$ and construct the graphs $F_{i,n}$.

\begin{figure}[ht]
    \centering
    \includegraphics[scale=0.6]{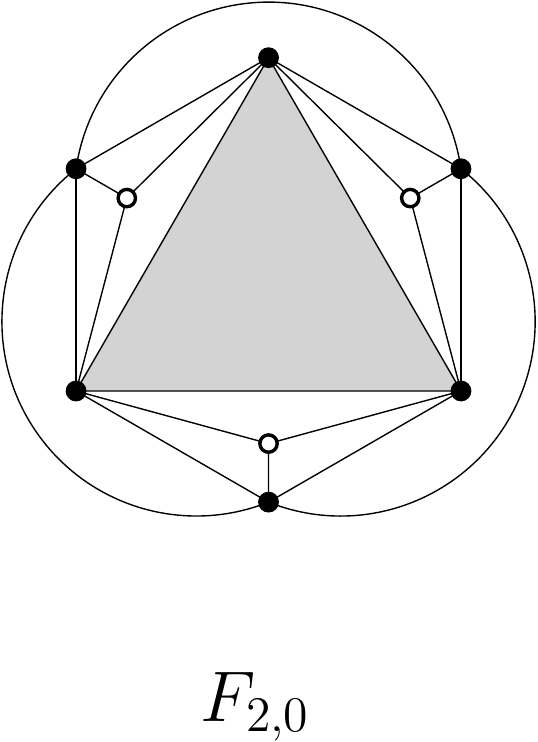}
    \caption{The construction of the graph $F_{2, 0}$.
    The highlighted triangle represents a subgraph $T$.} 
    \label{f:87}
\end{figure}

%%%%%%%%%%%%%%%%%%%%%%%%%%%%%%%%%%%%%%%%%%%%%%%%%%%%%%%%%%%%%%%%%%%%%%%%%%%%%%%%%%%%%%%%%%%%%%%
%%%%%%%%%%%%%%%%%%%%%%%%%%%%%%%%%%%%%%%%%%%%%%%%%%%%%%%%%%%%%%%%%%%%%%%%%%%%%%%%%%%%%%%%%%%%%%%

\section{Families of tree-like structured graphs}
\label{sc}

We recall the standard construction used for bounding the shortness exponent,
and we formalize it with the following definition and Lemma~\ref{cyc}.

An \emph{arranged block} is a $5$-tuple $(G_0, j, W, O, k)$
where $G_0$ is a graph, $j$ is the number of vertices of $G_0$,
and $W$ and $O$ are disjoint sets of vertices of $G_0$
such that the vertices of $W$ are simplicial and independent 
and $O$ induces a complete graph
and such that every cycle in $G_0$ contains at most $k$ vertices of $W$.

\begin{lemma}\label{cyc}
Let $(G_0, j, W, O, k)$ be an arranged block such that $k \geq 1$. 
For every $n \geq 1$, let $G_n$ be a graph obtained from $G_{n-1}$
by replacing every vertex of $W$ with a copy of $G_0$ (which contains $W$ and $O$),
and by adding arbitrary edges which connect the neighbourhood of the replaced vertex
with the set $O$ of the copy of $G_0$ replacing this vertex.
Then $G_n$ has $1 + (j - 1)(1 + |W| + \dots + |W|^n)$ vertices and its longest cycle has at most
$1 + (\ell - 1)(1 + k + \dots + k^n)$
vertices where $\ell = j - |W| + k$.
\end{lemma}
\begin{proof}
We note that $G_n$ contains $|W|^{n+1}$ vertices of $W$.
For the sake of induction, we prove the statement with an additional claim
that every cycle in $G_n$ contains at most $k^{n+1}$ vertices of $W$.
Clearly, the statement and the claim are satisfied for $n = 0$, and we proceed by induction on $n$.

We note that the difference in the order of $G_n$ and $G_{n-1}$ equals $(j - 1)|W|^n$.
Thus, $G_n$ has $1 + (j - 1)(1 + |W| + \dots + |W|^n)$ vertices.

We let $C$ be a cycle in $G_n$, and we view this cycle simply as a sequence of vertices.
For every newly added copy of $G_0$, we remove from $C$ all but one vertex of the copy
and we replace the remaining vertex (if there is such) by the corresponding replaced vertex of $G_{n-1}$;
and we let $C'$ denote the resulting sequence.
Clearly, if $C'$ has at most two vertices, then $C$ visits at most one of the newly added copies of $G_0$.
If $C'$ has at least three vertices,
then $C'$ defines a cycle in $G_{n-1}$
(since the neighbourhood of every vertex of $W$ in $G_{n-1}$ induces a complete graph);
and $C'$ contains at most $k^n$ vertices of $W$ (by the induction hypothesis).
Thus, $C$ visits at most $k^n$ of the newly added copies of $G_0$.

Similarly, we choose an arbitrary newly added copy of $G_0$, and
we remove from $C$ all vertices not belonging to this copy.
We note that the resulting sequence either contains at most two vertices (belonging to $O$)
or it is a cycle in $G_0$ (since $O$ induces a complete graph).
Thus, a cycle in $G_n$ contains at most $k$ vertices belonging to $W$ of one copy of $G_0$.
Furthermore, a cycle contains at most $j - |W| + k$ vertices of one such copy.

Consequently, $C$ contains at most $k^{n+1}$ vertices of $W$.
Furthermore, the length of $C$ minus the length of a longest cycle in $G_{n-1}$
is at most $(j - |W| + k -1)k^n$ which concludes the proof.
\end{proof}

For $i = 1,2$, we consider this construction for the graph $F_{i,0}$ playing the role of $G_0$
and the set of its white vertices playing the role of $W$
and the set of vertices of its outer face playing the role of $O$.
For every added copy of $F_{i,0}$, we join the vertices of its outer face
to the neighbourhood of the corresponding replaced vertex
by adding six edges in such a way that the new edges form a $2$-regular bipartite graph
(that is, a cycle of length $6$).
We let $F_{i, n}$ be the resulting graphs, and we observe that they are maximal planar graphs.
For instance, see the graph $F_{2, 1}$ depicted in Figure~\ref{f:F21}.

\begin{figure}[ht]
    \centering
    \includegraphics[scale=0.6]{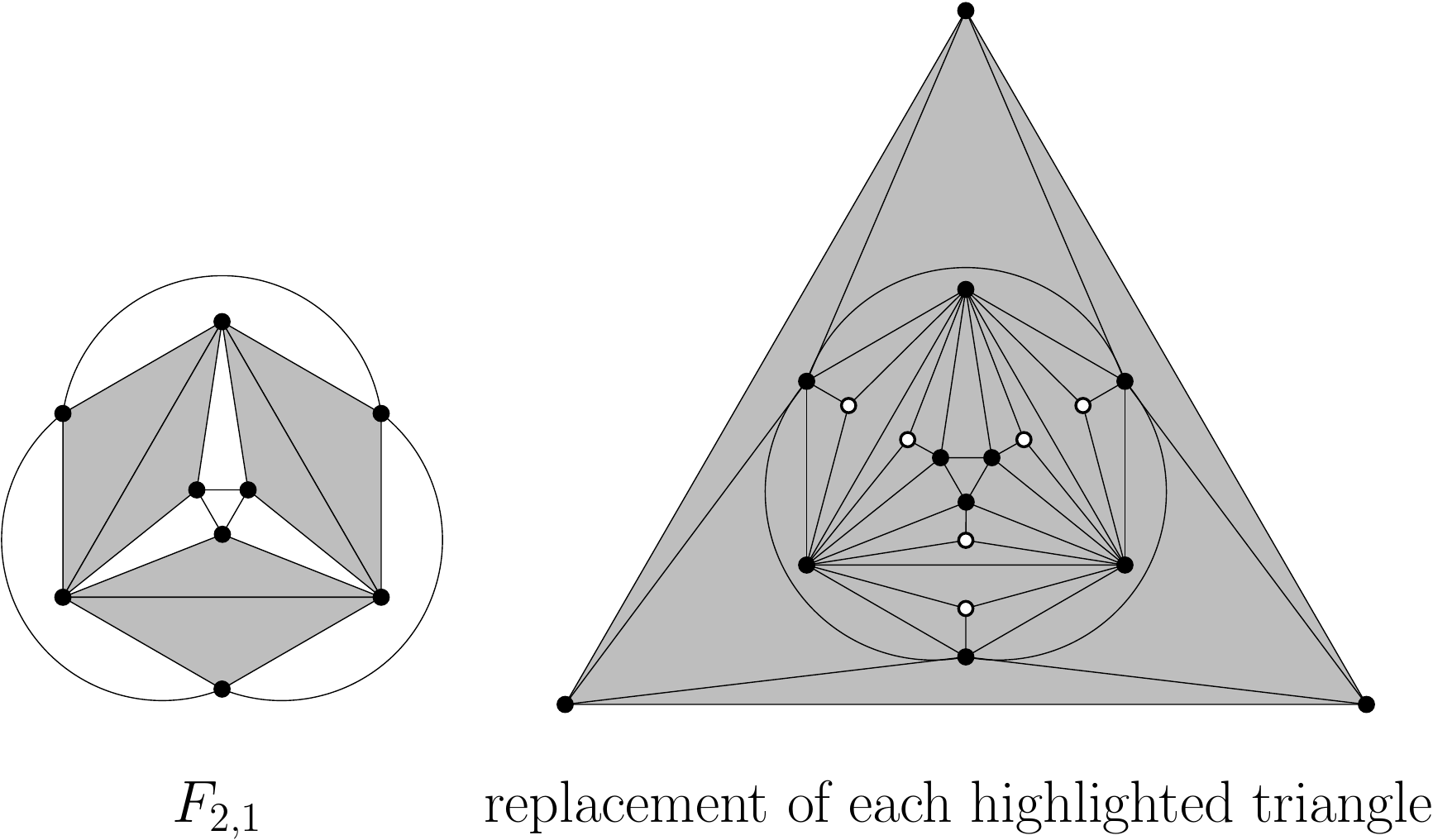}
    \caption{The construction of the graph $F_{2, 1}$.
    The graph $F_{2, 1}$ is obtained from the smaller graph (left) by replacing each of its highlighted triangles
    with the larger graph (right) in the natural way.
    (This corresponds to replacing white vertices of $F_{2, 0}$ with copies of $F_{2, 0}$
    and adding edges as in the present construction.)} 
    \label{f:F21}
\end{figure}
We start verifying that the constructed graphs (for $i = 1,2$) have the desired properties. 

\begin{coro}\label{cyc123}
For every $i = 1,2$ and every non-negative integer $n$,
the graph $F_{i,n}$ has $f_i(n)$ vertices
and its longest cycle has
$c_i(n)$ vertices where 
\begin{enumerate}
\item $f_1(n) = 1 + 101(1 + 30 + \dots + 30^n)$ and $c_1(n) = 1 + 93(1 + 22 + \dots + 22^n)$,
\item $f_2(n) = 1 + 14(1 + 6 + \dots + 6^n)$ and $c_2(n) = 1 + 13(1 + 5 + \dots + 5^n)$.
\end{enumerate}
\end{coro}
\begin{proof}
The order of the graphs
and the upper bound on the length of their longest cycles follow from Lemma~\ref{cyc}.

We note that a longest cycle in $F_{1,0}$, $F_{2,0}$ has $94$, $14$ vertices, respectively. 
Furthermore, there is a longest cycle which contains an edge of the outer face.
Clearly, by removing this edge from the cycle we obtain a path whose ends are vertices of the outer face.
We consider a longest cycle in $F_{i,n-1}$ and we extend it to a cycle in $F_{i,n}$ using these paths.
The following observation shows that such an extension is possible.
For an arbitrary pair, say $A$, of neighbours of one replaced vertex
and an arbitrary pair, say $B$, of vertices of the outer face of the corresponding $F_{i,0}$ (used for replacing this vertex),
the bipartite graph $(A, B)$ has a perfect matching.
  
A simple counting argument gives that $F_{i,n}$ contains a cycle of length $c_i(n)$,
for every $i = 1,2$ and every $n$.
\end{proof}

We use a slightly different construction to get the graphs $F_{3, n}$.
The building block $F_{3,0}$ is the graph $T$ whose simplicial vertices are coloured white.
We view a subgraph induced by a white vertex and its neighbourhood as a
\emph{$K_4$-region}, that is, a subgraph isomorphic to $K_4$
whose white vertex has degree $3$ in the whole graph, and the neighbours of
the white vertex are called \emph{outer} vertices of the $K_4$-region.

For $n \geq 1$, we let $F_{3, n}$ be a graph obtained from $F_{3, n-1}$
by replacing every $K_4$-region of $F_{3, n-1}$ with a $T$-region
in the natural way (the outer vertices of
the $K_4$-region are the outer vertices of
the $T$-region);
and we note that they are maximal planar graphs.
For instance, the graph $F_{3, 1}$ is depicted in Figure~\ref{f:F31}.
We proceed by the following proposition.

\begin{figure}[ht]
    \centering
    \includegraphics[scale=0.6]{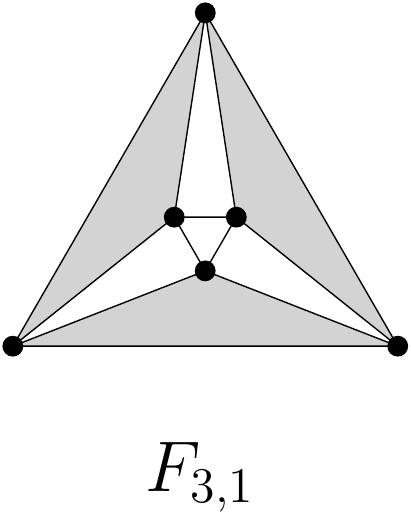}
    \caption{The construction of the graph $F_{3, 1}$.
    Each highlighted triangle represents a subgraph $T$.} 
    \label{f:F31}
\end{figure}

\begin{prop}\label{cyc4}
For every non-negative integer $n$, the graph $F_{3,n}$
has $4 + 5(1 + 3 + \dots + 3^n)$ vertices
and its longest cycle has $3 \cdot 2^{n+3}-9n-15$ vertices.
\end{prop}
\begin{proof}
We verify the order of $F_{3,n}$ using induction on $n$.
Clearly, $F_{3,0}$ has $9$ vertices,
and we note that the difference in the order of $F_{3,n}$ and $F_{3,n-1}$ equals $5 \cdot 3^n$.
Thus, $F_{3,n}$ has $4 + 5(1 + 3 + \dots + 3^n)$ vertices.

We show the length of a longest cycle using a slightly technical argument.
We let $s_i(n)$ denote the length of a longest cycle in $F_{3,n}$
which contains
$i$ edges of the outer face
(in the embedding which follows naturally from the construction).
For the sake of induction, we prove the following equalities:
\begin{align}
\label{eq:1}
s_0(n) & = 3s_1(n-1) -3 \nonumber \\
s_1(n) & = 2s_2(n-1) + s_1(n-1) -3 \\
s_2(n) & = 2s_2(n-1) \nonumber
\end{align}
and
\begin{align}
\label{eq:2}
s_0(n) & = 3 \cdot 2^{n+3}-9n -15 \nonumber\\
s_1(n) & = 2^{n+4} - 3n - 7 \\
s_2(n) & = 2^{n+3}. \nonumber
\end{align}

Clearly, $s_0(0) = s_1(0) = 9$ and $s_2(0) = 8$,
and (using Proposition~\ref{T}) we note that $s_0(1) = 24$, $s_1(1) = 22$ and $s_2(1) = 16$;
so the equalities are satisfied for $n = 1$.
We assume that they are satisfied for $n-1$ and we prove them for $n$.

For $n \geq 1$, we view $F_{3,n}$ as the graph
obtained from $F_{3,0}$ by replacing each of its
$K_4$-regions with an $F_{3,n-1}$-region;
and we view a cycle, say $C$, of $F_{3,n}$ as a sequence of vertices.
We consider one of the $F_{3,n-1}$-regions, say $R$,
and we remove from $C$ all vertices not belonging to $R$;
and we let $C'$ be the resulting sequence.
We observe that $C'$ either is a cycle or it has at most two vertices.
Furthermore, if $C$ visits a vertex not belonging to $R$,
then $C'$ (if it has at least three vertices)
is a cycle in $R$ containing at least one edge of its outer face.

Clearly, a longest cycle (in $F_{3,n}$) whose all vertices belong to the same $F_{3,n-1}$-region
has $s_0(n-1)$ vertices, and we observe that a longest cycle
visiting more than one of the regions has $3s_1(n-1) -3$ vertices.
We use~\eqref{eq:2} for $s_0(n-1)$ and $s_1(n-1)$,
and we note that 
\begin{align*}
3 \cdot 2^{n+2}-9(n-1)-15 < 3(2^{n+3}- 3(n-1) - 7) -3.
\end{align*}
So we get
$s_0(n) = 3s_1(n-1) -3$.
By similar arguments, we get
$s_1(n) = 2s_2(n-1) + s_1(n-1) -3$
and
$s_2(n) = 2s_2(n-1)$.

Consequently, we can use~\eqref{eq:1} for $s_i(n)$, and we obtain 
\begin{align*}
s_0(n) =  3s_1(n-1)-3 =  3(2^{n+3} - 3(n-1) - 7) - 3 = 3 \cdot 2^{n+3}-9n-15
\end{align*}
and
\begin{align*}
s_1(n) &= 2s_2(n-1) + s_1(n-1) -3 = 2 \cdot 2^{n+2} + 2^{n+3} - 3(n-1) - 7 - 3 \\
&= 2^{n+4} - 3n - 7
\end{align*}
and
\begin{align*}
s_2(n) = 2s_2(n-1) = 2 \cdot 2^{n+2} = 2^{n+3}.
\end{align*}
Thus, the equalities are satisfied for $n$.
In particular, a longest cycle in $F_{3,n}$ has $3 \cdot 2^{n+3}-9n-15$ vertices.
\end{proof}

With Corollary~\ref{cyc123} and Proposition~\ref{cyc4} on hand,
we shall focus on the toughness of the constructed graphs. 

%%%%%%%%%%%%%%%%%%%%%%%%%%%%%%%%%%%%%%%%%%%%%%%%%%%%%%%%%%%%%%%%%%%%%%%%%%%%%%%%%%%%%%%%%%%%%%%
%%%%%%%%%%%%%%%%%%%%%%%%%%%%%%%%%%%%%%%%%%%%%%%%%%%%%%%%%%%%%%%%%%%%%%%%%%%%%%%%%%%%%%%%%%%%%%%

\section{Toughness of the extended blocks}
\label{tb}

In this section, we study the toughness of $F_{i,0}$ (for $i = 1,2,3$)
and of its extension $F^+_{i,0}$ (for $i = 1,2$) which is
a graph obtained by adding a vertex adjacent to all vertices of the outer face
of $F_{i,0}$.
We shall use the following two propositions.

\begin{prop}\label{simpl}
Adding a simplicial vertex to a graph does not increase its toughness.
\end{prop}
\begin{proof}
We let $x$ be a simplicial vertex of a graph $G^+$ and we let $G = G^+ - x$
and we let $S$ be a set of vertices of $G$.
We note that
$c(G^+-S) \geq  c(G-S)$,
and the statement follows.
\end{proof}

\begin{prop}\label{Ttough}
Let $R$ be a $T$-region of a graph $G$
and let $I$, $O$ be the set of all inner, outer vertices of $R$, respectively.
Let $t \geq 1$
and let $S$ be a set of vertices of $G$
such that $c(G - S) > \frac{1}{t}|S|$.
If $|S \cap O| \geq 2$,
then there is a separating set $S' = (S \sm I) \cup A$
such that $c(G - S') > \frac{1}{t}|S'|$
where $A$ is chosen as follows.

\begin{itemize}
\item If $|S \cap O| = 2$, then $A$ consists of
the non-simplicial vertex of $I$ which is the common neighbour
of the vertices of $S \cap O$.
\item If $|S \cap O| = 3$,
then $A$ consists of two non-simplicial vertices of $I$.
\end{itemize}
\end{prop}
\begin{proof}
In both cases, we modify $S$ as suggested; and we let $S'$ be the resulting set.
Clearly, $S'$ is a separating set and $c(G-S') - c(G-S) \geq 0$, and we observe that 
\begin{equation*}
c(G-S') - c(G-S) \geq |S'|- |S|.
\end{equation*}
Since $t \geq 1$, we have either $0 > \frac{1}{t}(|S'|- |S|)$ or $|S'|- |S| \geq \frac{1}{t}(|S'|- |S|)$.
Consequently, we obtain 
\begin{equation*}
c(G-S') - c(G-S) \geq \tfrac{1}{t}|S'|- \tfrac{1}{t}|S|.
\end{equation*}
We use $c(G - S) > \frac{1}{t}|S|$ and we conclude that $c(G - S') > \frac{1}{t}|S'|$.
\end{proof}

We recall that a set $D$ of vertices is \emph{dominating} a graph $G$
if every vertex of $G - D$ is adjacent to a vertex of $D$.
We note that every pair of vertices of degree $6$ is dominating $T$,
so every separating set in $T$ has at least two of these vertices.
As a consequence of Proposition~\ref{Ttough}, we note the following.

\begin{coro}\label{32tough}
The graph $T$ is $\frac{3}{2}$-tough.
\end{coro}

Furthermore, if a graph contains a $T$-region 
and a vertex not belonging to the $T$-region, then the toughness of the graph is at most $\frac{5}{4}$;
and we show that this is the correct value for $F^+_{1, 0}$ and $F_{1, 0}$.

\begin{prop}\label{54tough}
The graphs $F^+_{1, 0}$ and $F_{1, 0}$ are $\frac{5}{4}$-tough.
\end{prop}
\begin{proof}
By Proposition~\ref{simpl}, it suffices to show that $F^+_{1, 0}$ is $\frac{5}{4}$-tough.
To the contrary, we suppose that
there is a separating set $S$ of vertices
such that $c(F^+_{1, 0} - S) > \frac{4}{5}|S|$. 
We consider such $S$ adjusted by using Proposition~\ref{Ttough},
in sequence, for all $T$-regions of $F^+_{1, 0}$.

We let $\II$ denote the set of all components of $F^+_{1, 0} - S$
consisting exclusively of inner vertices of some $T$-region.
We note that the existence of such component implies that
at least two outer vertices of the corresponding $T$-region belong to $S$
(since every pair of outer vertices of $T$ is dominating $T$).
We let $r_2$, $r_3$ denote the number of $T$-regions whose exactly $2$, $3$
outer vertices belong to $S$, respectively.
We let $c$ denote the common vertex of all $T$-regions.
Considering the inner vertices of the $T$-regions,
we call simplicial such vertices \emph{white}
and the remaining such vertices \emph{grey}.
Except for $c$, the outer vertices of the $T$-regions are called \emph{black}.
The vertices adjacent to a black vertex
but not belonging to a $T$-region are called \emph{blue}.
We let $c'$ denote the vertex adjacent to all blue vertices
and $x$ denote the vertex of $F^+_{1, 0}$ not belonging to $F_{1, 0}$.
We let $\BB$ denote the set of all components of $F^+_{1, 0} - S$
containing a black vertex or a blue vertex.

We shall use a discharging argument to avoid complicated inequalities.
We assign charge $5$ to every component of $F^+_{1, 0} - S$,
and we aim to distribute all assigned charge among the vertices of $S$,
and to show that every vertex of $S$ receives charge at most $4$,
contradicting the assumption that $c(F^+_{1, 0} - S) > \frac{4}{5}|S|$.
We pre-distribute the charge according to the following rules.
\begin{itemize}
\item Every grey vertex of $S$ receives $4$
of the total charge of the components of $\II$ belonging to the same $T$-region as the grey vertex.
\item For every $T$-region, the remaining charge
of all components of $\II$ belonging to this $T$-region is distributed
equally among the black vertices of $S$ belonging to this $T$-region.
\item Every blue vertex of $S$ receives
as much of the total charge of components outside $\II$ as possible (at most $4$).
\end{itemize}
We note that after the pre-distribution, the charge of every component of $\II$ is $0$;
and we focus on the remaining charge of the rest of the components.

If $c(F^+_{1, 0} - S) - |\II| \leq 1$, then we have $|\II| \geq 1$ and the remaining charge is at most $5$.
Consequently, $r_2 \geq 1$ or $r_3 \geq 1$,
and in both cases, the vertices of $S$ can still receive charge at least $5$, a contradiction.

We assume that $c(F^+_{1, 0} - S) - |\II| \geq 2$.
We show that $F^+_{1, 0} - S$ has no component containing $c$ and no black vertex.
The existence of such component implies that
all black vertices belong to $S$, and these vertices can still receive $\frac{7}{2} \cdot 20$.
A contradiction follows by counting the maximum possible number of components not belonging to $\II$.

Consequently, $F^+_{1, 0} - S$ has at most one component belonging to neither $\II$ nor $\BB$.
If $|\BB| \leq 1$, then there is such component (since $c(F^+_{1, 0} - S) - |\II| \geq 2$).
Clearly, this component contains $c'$ or $x$, that is, 
all blue vertices or $c'$ and at least two blue vertices belong to $S$,
and a contradiction follows.

We assume that $|\BB| \geq 2$.
On the other hand, considering the graph induced by black and blue vertices,
we observe that the size of a maximum independent set
of this graph is $13$. Thus, $|\BB| \leq 13$.

Since $|\BB| \geq 2$, we note that at least $|\BB|$ black and at least $|\BB|$ blue vertices belong to $S$.
We let $d$ denote the number of blue vertices of $S$ minus $|\BB|$.
We recall that $F^+_{1, 0} - S$ has at most $|\BB| + 1$ components not belonging to $\II$.
Thus, the remaining charge, which is yet to be distributed, is at most $|\BB| + 5 - 4d$;
and since $d \geq 0$, it is at most $|\BB| + 5$.

If $c$ does not belong to $S$, then 
the black vertices of $S$ can still receive at least $\frac{7}{2}|\BB|$.
Clearly, $\frac{7}{2}|\BB| \geq |\BB| + 5$ since $|\BB| \geq 2$, a contradiction.

We assume that $c$ belongs to $S$. 
Clearly, $c$ can receive $4$.
We note that the black vertices of $S$ can still receive $3r_2 + r_3$;
and since $r_2 + 2r_3 \geq |\BB|$, they can receive at least $\frac{1}{2}|\BB| + \frac{5}{2}r_2$
(thus, at least $\frac{1}{2}|\BB|$).

If $c'$ does not belong to $S$,
then we consider the graph induced by $c'$ and by the black and blue vertices,
and we observe that $r_2 + d \geq |\BB| - 1$
(since there are at least $|\BB| - 1$ components of $\BB$ not containing $c'$).
Consequently, we have $\frac{1}{2}|\BB| + \frac{5}{2}r_2 + 4 > |\BB| + 5 - 4d$ since $|\BB| \geq 2$, a contradiction.

We assume that $c'$ belongs to $S$,
so $c'$ can receive $4$.
If $d \geq 1$, then
the remaining charge is at most $|\BB| + 1$;
and we have $\frac{1}{2}|\BB| + 4 + 4 > |\BB| + 1$ since $|\BB| \leq 13$, a contradiction.

We assume that $d = 0$.
If there is not a component containing $x$ as its only vertex,
then the remaining charge is $|\BB|$.
Similarly as above, we have $\frac{1}{2}|\BB| + 4 + 4 > |\BB|$, a contradiction.

We assume that $F^+_{1, 0} - S$ has a component consisting of $x$,
so all neighbours of $x$ belong to $S$.
Since $d = 0$ and since $|\BB| \geq 2$, we have $r_2 \geq 1$.
If $|\BB| \leq 11$, then we have $\frac{1}{2}|\BB| + \frac{5}{2}r_2 + 4 + 4 \geq |\BB| + 5$, a contradiction.
If $|\BB| = 12$, then the parity implies that $r_2 \geq 2$ or $r_3 \geq 6$, and a contradiction follows.

We assume that $|\BB| = 13$.
We consider the structure of $\BB$ and we observe that $r_3 \leq 4$, that is, $r_2 \geq 5$.
Consequently, we get $\frac{1}{2}|\BB| + \frac{5}{2}r_2 + 4 + 4 > |\BB| + 5$,
and we obtain the desired distribution of the assigned charge, a contradiction.
Thus, $F^+_{1, 0}$ is $\frac{5}{4}$-tough. 
\end{proof}

We continue with the following.

\begin{prop}\label{87tough}
The graphs $F^+_{2, 0}$ and $F_{2, 0}$ are $\frac{8}{7}$-tough.
\end{prop}
\begin{proof}
By Proposition~\ref{simpl}, it suffices to show
the toughness of $F^+_{2, 0}$ which we do via contradiction. 
We suppose that
there is a separating set $S$ of vertices in $F^+_{2, 0}$
such that $c(F^+_{2, 0} - S) > \frac{7}{8}|S|$. 

Since $S$ is separating, we have $|S| \geq 3$.
Consequently, we can assume that $c(F^+_{2, 0} - S) \geq 3$.  

To specify the structure of $S$, we consider the graph $F_{2, 0}$.
We note that $F_{2, 0}$ 
contains two $T$-regions and the $T$-regions share their outer vertices;
and we call these outer vertices \emph{black}
(as well as the corresponding vertices of $F^+_{2, 0}$).
We observe that every pair of black vertices is dominating $F_{2, 0}$.

Since $c(F^+_{2, 0} - S) \geq 3$,
we have that at least two black vertices belong to $S$.
We note that $F^+_{2, 0}$ contains one $T$-region
(and its outer vertices are black),
and we modify $S$ using Proposition~\ref{Ttough};
and we let $S'$ be the resulting set.
Considering the possibilities,
we observe that $c(F^+_{2, 0} - S') \leq \frac{7}{8}|S'|$, a contradiction.
\end{proof}

We shall use the toughness of the building blocks $F_{i, 0}$
(given by Propositions~\ref{54tough} and~\ref{87tough} and by Corollary~\ref{32tough})
to show the toughness of the constructed graphs $F_{i, n}$.

%%%%%%%%%%%%%%%%%%%%%%%%%%%%%%%%%%%%%%%%%%%%%%%%%%%%%%%%%%%%%%%%%%%%%%%%%%%%%%%%%%%%%%%%%%%%%%%
%%%%%%%%%%%%%%%%%%%%%%%%%%%%%%%%%%%%%%%%%%%%%%%%%%%%%%%%%%%%%%%%%%%%%%%%%%%%%%%%%%%%%%%%%%%%%%%

\section{Gluing tough graphs}
\label{sg}

We shall use the following lemma as the main tool for showing the toughness of graphs
which are obtained by the standard construction for bounding the shortness exponent. 

\begin{lemma}\label{constr2}
For $i = 1,2$, let $G^+_i$ and $G_i$ be $t$-tough graphs
such that $G_i$ is obtained by removing vertex $v_i$ from $G^+_i$.
Let $U$ be a graph obtained from the disjoint union of $G_1$ and $G_2$ by adding new edges 
such that the minimum degree of the bipartite graph $(N(v_1), N(v_2))$ is at least $t$.
Then $U$ is $t$-tough.
\end{lemma}
\begin{proof}
We assume that $t > 0$ and that there exists a separating set of vertices in $U$.
We let $X$ be such a set and we let $X_i = X \cap V(G_i)$ for $i = 1,2$.
Clearly, $2 \leq c(U-X) \leq c(G_1-X_1) + c(G_2-X_2)$,
and we use it to show that $c(U-X) \leq \frac{1}{t}|X|$.

If $X_i$ is a separating set in $G_i$, then the toughness of $G_i$ implies that
$c(G_i-X_i) \leq \frac{1}{t}|X_i|$.

We suppose that $X_1$ is not separating in $G_1$,
and we observe that $c(U- X) \leq c(G_2-X_2) +1$.
If $c(U- X) \leq c(G_2-X_2)$, then $X_2$ is separating in $G_2$,
and thus, $c(G_2-X_2) \leq \frac{1}{t}|X_2|$
and the desired inequality follows since $|X_2| \leq |X|$.

We assume that $c(U- X) = c(G_2-X_2) +1$.
Clearly, if $N(v_2) \subseteq X_2$, then $c(G_2-X_2) +1 = c(G^+_2-X_2)$;
so $X_2$ is separating in $G^+_2$ and
we have $c(G^+_2-X_2) \leq \frac{1}{t}|X_2|$ and the inequality follows.

In addition, we assume that there is a vertex of $N(v_2)$ not belonging to $X_2$.
We recall that this vertex has at least $\lceil t \rceil$ neighbours in $N(v_1)$.
Since $c(U- X) = c(G_2-X_2) +1$, we note that all these neighbours belong to $X_1$.
Thus, $|X_1| \geq t$ and we have $c(G_1-X_1) \leq \frac{1}{t}|X_1|$.

Similarly, we get that if $X_2$ is not separating in $G_2$,
then $c(G_2-X_2) \leq \frac{1}{t}|X_2|$.
We conclude that (no matter whether $X_i$ is separating or not)
we have $c(G_i-X_i) \leq \frac{1}{t}|X_i|$ for both $i = 1,2$;
and the inequality follows.
\end{proof}

\begin{figure}[ht]
    \centering
    \includegraphics[scale=0.6]{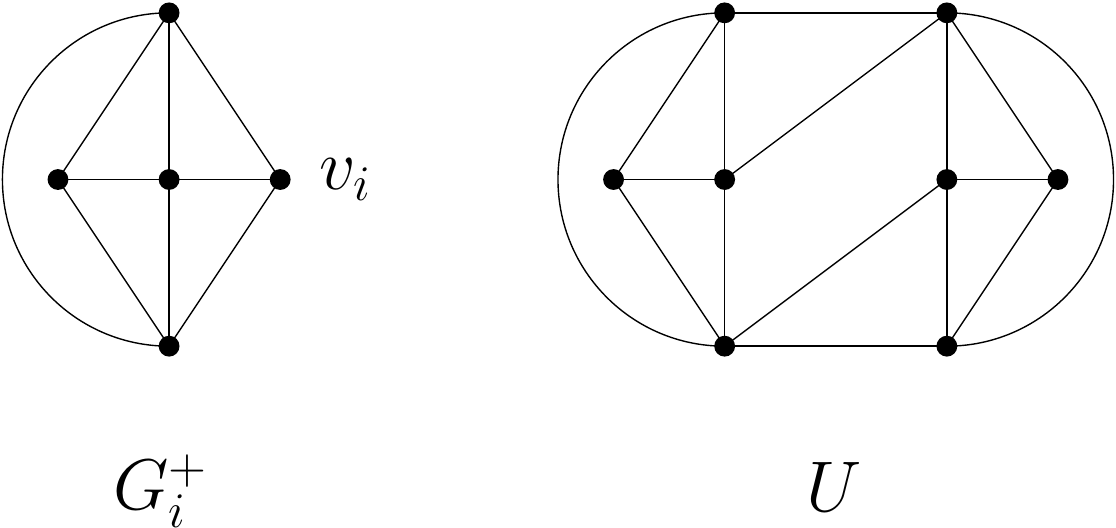}
    \caption{A counterexample to a statement presented in~\cite[Lemma~1]{5/4}.
    We consider the graphs $G^+_i$ and $G_i = G^+_i - v_i$ for $i = 1,2$, and the graph $U$.
    We note that $U$ is obtained from the disjoint union of $G_1$ and $G_2$
    by adding edges and every vertex of $N(v_1) \cup N(v_2)$ is incident with at least one new edge.
    Clearly, $G^+_i$ and $G_i$ are $\frac{3}{2}$-tough,
    but $U$ is not.} 
    \label{f:counter}
\end{figure}

We remark that a similar statement appears in~\cite[Lemma~1]{5/4};
and it is used in~\cite{chpl}.
(Compared to Lemma~\ref{constr2},
the main difference is that the minimum degree
of the considered bipartite graph is required to be at least $1$.)
The graphs depicted in Figure~\ref{f:counter} show that this statement is false.
We view Lemma~\ref{constr2} as a fixed version of this statement,
and we remark that Lemma~\ref{constr2} can be applied in the arguments of~\cite{5/4} and~\cite{chpl}. 
We note that Lemma~\ref{constr2} can be viewed as a generalization of a similar statement
(for $1$-tough graphs) presented in~\cite[Lemma~4]{dill}.

%%%%%%%%%%%%%%%%%%%%%%%%%%%%%%%%%%%%%%%%%%%%%%%%%%%%%%%%%%%%%%%%%%%%%%%%%%%%%%%%%%%%%%%%%%%%%%%
%%%%%%%%%%%%%%%%%%%%%%%%%%%%%%%%%%%%%%%%%%%%%%%%%%%%%%%%%%%%%%%%%%%%%%%%%%%%%%%%%%%%%%%%%%%%%%%

\section{Toughness of the constructed graphs}
\label{tt}

In this section, we clarify that the graphs $F_{i, n}$ have the properties
stated in Proposition~\ref{shortAll}.
\begin{proof}[Proof of Proposition~\ref{shortAll}]
We recall that the order of the constructed graphs and the length of their longest cycles are given
by Corollary~\ref{cyc123} and by Proposition~\ref{cyc4}.

For every $i = 1,2$, we show the toughness of the graphs $F_{i, n}$ using induction on $n$.
The case $n = 0$ is verified by Propositions~\ref{54tough} and~\ref{87tough}.
By induction hypothesis, $F_{i, n-1}$ has the required toughness,
and by Proposition~\ref{simpl}, so does a graph obtained from $F_{i, n-1}$ by removing a simplicial vertex;
we shall apply Lemma~\ref{constr2}, and we view these two graphs as playing the role of $G^+_1$ and $G_1$ and 
we view graphs $F^+_{i, 0}$ and $F_{i, 0}$ as playing the role of $G^+_2$ and $G_2$.
We consider the graph obtained from $F_{i, n-1}$ by replacing one 
of its simplicial vertices by a copy of $F_{i, 0}$ and by adding edges as in the present construction
(we recall the construction of graphs $F_{i, n}$ for $i = 1,2$; see Section~\ref{sc}).
By Lemma~\ref{constr2}, the resulting graph has the required toughness.
Thus, we can replace a simplicial vertex
of the resulting graph and apply Lemma~\ref{constr2} again;
and repeating this argument, we obtain that $F_{i, n}$ has the required toughness.

Similarly, we show the toughness of $F_{3, n}$ by induction on $n$.
By Corollary~\ref{32tough}, the toughness of $F_{3, 0}$ is greater than $1$.
We consider Lemma~\ref{constr3} (see below)
applied on the graph $F_{3, n-1}$ playing the role of $G$
(and then applied repeatedly on the resulting graph),
and we obtain that the toughness of $F_{3, n}$ is greater than $1$.
\end{proof}

Similarly to Lemma~\ref{constr2}, the following lemma can be used to
construct large graphs from smaller ones while preserving certain toughness.

\begin{lemma}\label{constr3}
Let $G$ be a graph of toughness greater than $1$ which contains a $K_4$-region
and let $G'$ be a graph obtained from $G$
by replacing this $K_4$-region by a $T$-region (in the natural way).
Then the toughness of $G'$ is greater than $1$.
\end{lemma}
\begin{proof}
We let $X'$ be a separating set of vertices in $G'$, and we shall show that $c(G'-X') < |X'|$.
We consider the set $X$ obtained from $X'$ by removing all inner vertices of the new $T$-region.
For the sake of simplicity,
we let $c = c(G' - X') - c(G - X)$ and $x = |X' \sm X|$,
and we note that it suffices to show that $c(G-X)+c < |X|+x$.
Considering the choice of $X$, we observe that $c \leq x$.
We conclude the proof by showing that $c(G-X) < |X|$.
If $X$ is separating in $G$, then the inequality is given by the toughness of $G$.
We can assume that $c(G-X) = 1$. Consequently, $c(G'-X') > c(G-X)$,
so at least two outer vertices of the new $T$-region belong to $X'$.
Thus, they belong to $X$, that is, $|X| \geq 2$ and the inequality follows.
\end{proof}

\section{Note on longest paths}
\label{lp}

We remark that (using similar arguments as in Section~\ref{sc} we obtain that)
a longest path of $F_{i,n}$ has $p_i(n)$ vertices where
\begin{enumerate}
\setlength{\itemsep}{3pt}
\item $p_1(n) = 2 + c_1(n) + 2\sum_{k=0}^{n-1} c_1(k)$,
\item $p_2(n) = 1 + \text{sgn}(n) + c_2(n) + c_2(n-1) + 2\sum_{k=0}^{n-2} c_2(k)$,
\item $p_3(n) = 7 \cdot 2^{n+2} + 2\text{sgn}(n) -15n -19$.
\end{enumerate}

\section*{Acknowledgements}

The author would like to thank Petr Vr\'{a}na for discussing different
strategies for the graph constructions,
and to thank the anonymous referees for their helpful comments.   

The research was supported by the project LO1506 of the Czech Ministry of Education, Youth and Sports
and by the project 17-04611S of the Czech Science Foundation.

%%%%%%%%%%%%%%%%%%%%%%%%%%%%%%%%%%%%%%%%%%%%%%%%%%%%%%%%%%%%%%%%%%%%%%%%%%%%%%%%%%%%%%%%%%%%%%%
%%%%%%%%%%%%%%%%%%%%%%%%%%%%%%%%%%%%%%%%%%%%%%%%%%%%%%%%%%%%%%%%%%%%%%%%%%%%%%%%%%%%%%%%%%%%%%%

\end{document}